\documentclass[11pt]{article}

\usepackage{amsmath, amssymb, amsthm} 
\usepackage{caption}
\usepackage{dsfont}
\usepackage{amsthm}
\usepackage{amsmath}
\usepackage{amssymb}
\usepackage{setspace}
\usepackage{mathtools}
\usepackage{graphicx}
\graphicspath{ {./images/} }
\usepackage[hidelinks]{hyperref}
\usepackage{cleveref}
\usepackage{hyperref}
\usepackage{enumitem}
\usepackage{framed}
\usepackage{subcaption}
\usepackage{xspace}
\usepackage{thmtools} 
\usepackage{thm-restate}
\usepackage{tikz}
\usetikzlibrary{calc}
\usetikzlibrary{arrows,shapes,positioning, shapes.misc}
\tikzset{cross/.style={cross out, draw=black, minimum size=2*(#1-\pgflinewidth), inner sep=0pt, outer sep=0pt, line width=2pt},cross/.default={10pt}}
\tikzstyle{vertex}=[circle,draw=black,fill=black,inner sep=0,minimum size=5pt,text=white,font=\footnotesize]
\definecolor{darkgreen}{rgb}{0.1,0.7,0.1}

\allowdisplaybreaks
\expandafter\let\expandafter\oldproof\csname\string\proof\endcsname
\let\oldendproof\endproof
\renewenvironment{proof}[1][\proofname]{%
	\oldproof[\bf #1]%
}{\oldendproof}

\newcommand{\dist}{\mathsf{dist}}

\newcommand{\eps}{\varepsilon}
\allowdisplaybreaks

\theoremstyle{plain}
\newtheorem{theorem}{Theorem}[section]
\newtheorem{lemma}[theorem]{Lemma}

\newtheorem{conjecture}[theorem]{Conjecture}
\newtheorem{question}[theorem]{Question}

\theoremstyle{definition}

\newtheorem{definition}[theorem]{Definition}

\title{Restricted subgraphs of edge-colored graphs and applications}

\author{
Benny Sudakov\thanks{Department of Mathematics, ETH Z\"urich, Switzerland. Email: benjamin.sudakov@math.ethz.ch. Research supported in part by SNSF grant 200021-228014.}
}

\date{}

\begin{document}

\maketitle

\begin{abstract}
A properly edge-colored graph is a graph with a coloring of its edges such
that no vertex is incident to two or more edges of the same color. A subgraph
is called rainbow if all its edges have different colors. The problem of finding
rainbow subgraphs or other restricted structures in edge-colored graphs has
a long history, dating back to Euler's work on Latin squares. It has also proven
to be a powerful method for studying several well-known questions in other areas.

In this survey, we will provide a brief introduction to this topic, discuss
several results in this area, and demonstrate their applications to problems
in graph decomposition, additive combinatorics, theoretical computer science, and
coding theory.

\end{abstract}

\section{Introduction}
An edge-colored graph is a graph $G$ with assignment of colors to its edges. A subgraph $H$ of an edge-colored graph $G$ is called rainbow if all the edges of $H$ have different colors.
The study of rainbow subgraphs and  other restricted structures in edge-colored graphs goes back more than two hundred years to the work of Euler \cite{Euler} on Latin squares.
A Latin square of order  $n$ is an $ n \times n$  array filled with symbols from $\{1, \ldots, n\}$ such that each symbol appears once in every row and column. 
Well known examples of Latin squares are multiplication tables of finite groups. Latin squares have also connections to $2$-dimensional permutations, design theory, finite projective planes and error correcting codes (see, e.g., \cite{Wan, KPSY}). 
A pair of Latin squares $P=(p_{i,j})$ and $Q=(q_{i,j})$ are called mutually orthogonal if all the $n^2$ ordered pairs $(p_{i,j}, q_{i,j})$ are distinct. 
According to folklore, Catherine the Great asked Euler, who was residing at her court at  St.\ Petersburg at the end of 17th century the following question.
This question is known as the thirty-six officers problem and is equivalent to the existence of mutually orthogonal Latin squares of order $6$. Euler introduced it as follows:
\begin{quote}
``The question revolves around arranging 36 officers to be drawn from 6 different regiments so that they are ranged in a square so that in each line (both horizontal and vertical) there are 6 officers of different ranks and different regiments."
\end{quote}
Euler could not solve the problem, but was able to demonstrate the existence of the mutually orthogonal Latin squares when $n$ is odd or a multiple of $4$. Observing that no order two square exists and being unable to find a construction for $n=6$, he conjectured that no mutually orthogonal Latin squares of order $n$ exist for any even $n=2~(\!\!\!\!\mod 4)$

Euler's question is equivalent to asking for which values $n$ there is a Latin square which can be decomposed into $n$ disjoint transversals, where a transversal in a Latin square is a collection of cells which do not share the same row, column or symbol.
Indeed, given a Latin square $P=(p_{i,j})$ which is a disjoint union of transversals $T_1, \ldots, T_n$ one can construct another $ n \times n$ array $Q=(q_{i,j})$ by changing every entry in the traversal $T_i$ to be the symbol $i$. It is easy to check from the definition of transversal, that the resulting $Q$ is a Latin square which is orthogonal to $P$. 
The non-existence of orthogonal squares of order six was confirmed by Gaston Tarry \cite{Tarry} only in 1900. However, Euler's conjecture resisted solution for another 60 years until Parker, Bose, and Shrikhande \cite{BSP} in 1960 constructed mutually orthogonal Latin squares for all orders $n \geq 7$, showing that Euler's conjecture is false. 

Note that every Latin square $L=(\ell_{i,j})$ corresponds to an edge-coloring of the complete bipartite graph $K_{n,n}$ by coloring the edge $ij$ by the symbol $\ell_{i,j}$ from the corresponding cell $(i,j)$. It is easy to check that this gives a proper edge-coloring of $K_{n,n}$, using $n$ colors.
An edge-coloring is proper if no two edges sharing a vertex receive the same color. 
Identifying the cell $(i,j)$ with the edge $ij$, a
transversal in $L$ corresponds to a rainbow perfect matching (a collection of disjoint edges covering all the vertices), see Figure 1. Thus, finding transversals in Latin squares is a special case of finding rainbow subgraphs in properly edge-colored graphs. Another reason to study rainbow subgraphs arises in Ramsey theory, more precisely in the canonical version of Ramsey's theorem proved by Erd\H{o}s and Rado \cite{ErdosRado}. This theorem shows that locally-bounded edge-colorings of the complete graph $K_n$ contain rainbow copies of certain cliques. An edge-coloring is locally $k$-bounded if each vertex is in at most $k$ edges of any one color.
\begin{figure}[h]

\begin{minipage}{0.24\textwidth}
    \centering
\begin{tikzpicture}[scale=0.8]
    \fill[blue!80]  (0,0) rectangle (1,1); 
    \fill[red!80]   (0,1) rectangle (1,2); 
    \fill[green!80]  (0,2) rectangle (1,3); 

    \fill[green!80]  (1,0) rectangle (2,1); 
    \fill[blue!80]   (1,1) rectangle (2,2);  
    \fill[red!80]  (1,2) rectangle (2,3); 

    \fill[red!80]  (2,0) rectangle (3,1); 
    \fill[green!80]   (2,1) rectangle (3,2);  
    \fill[blue!80]  (2,2) rectangle (3,3); 

    \foreach \x in {0,1,2,3} {
        \draw (\x,0) -- (\x,3);
        \draw (0,\x) -- (3,\x);
    }

    \node at (0.5,3.5) {1};
    \node at (1.5,3.5) {2};
    \node at (2.5,3.5) {3};

    \node at (-0.5,2.5) {$a$};
    \node at (-0.5,1.5) {$b$};
    \node at (-0.5,0.5) {$c$};
\end{tikzpicture}

\end{minipage}
\begin{minipage}{0.24\textwidth}
\vspace{0.8cm}
    \centering
    \begin{tikzpicture}[scale=0.8]
    \node[vertex, label=above:$1$] (1) at (0, 2) {};
    \node[vertex, label=above:$2$] (2) at (2, 2) {};
    \node[vertex, label=above:$3$] (3) at (4, 2) {};

    \node[vertex, label=below:$a$] (a) at (0, 0) {};
    \node[vertex, label=below:$b$] (b) at (2, 0) {};
    \node[vertex, label=below:$c$] (c) at (4, 0) {};

    \draw[green, line width=0.5mm]  (a) -- (1);
    \draw[red, line width=0.5mm]    (b) -- (1);
    \draw[blue, line width=0.5mm]   (c) -- (1);
    
    \draw[red, line width=0.5mm]    (a) -- (2);
    \draw[blue, line width=0.5mm]   (b) -- (2);
    \draw[green, line width=0.5mm]  (c) -- (2);

    \draw[blue, line width=0.5mm]   (a) -- (3);
    \draw[green, line width=0.5mm]  (b) -- (3);
    \draw[red, line width=0.5mm]    (c) -- (3);
\end{tikzpicture}
\end{minipage}
\begin{minipage}{0.24\textwidth}
\centering
    \begin{tikzpicture}[scale=0.8]
    \fill[blue!80]  (0,0) rectangle (1,1); 
    \fill[red!80]   (0,1) rectangle (1,2); 
    \fill[green!80]  (0,2) rectangle (1,3); 

    \fill[green!80]  (1,0) rectangle (2,1); 
    \fill[blue!80]   (1,1) rectangle (2,2);  
    \fill[red!80]  (1,2) rectangle (2,3); 

    \fill[red!80]  (2,0) rectangle (3,1); 
    \fill[green!80]   (2,1) rectangle (3,2);  
    \fill[blue!80]  (2,2) rectangle (3,3); 

    \foreach \x in {0,1,2,3} {
        \draw (\x,0) -- (\x,3);
        \draw (0,\x) -- (3,\x);
    }

    \node at (0.5,3.5) {1};
    \node at (1.5,3.5) {2};
    \node at (2.5,3.5) {3};

    \node at (-0.5,2.5) {$a$};
    \node at (-0.5,1.5) {$b$};
    \node at (-0.5,0.5) {$c$};

    \draw (2.5, 0.5) node[cross,white] {};
    \draw (1.5, 1.5) node[cross,white] {};
    \draw (0.5, 2.5) node[cross,white] {};
    
\end{tikzpicture}
\end{minipage}
\begin{minipage}{0.24\textwidth}
\vspace{0.8cm}
\centering
    \begin{tikzpicture}[scale=0.8]
    \node[vertex, label=above:$1$] (1) at (0, 2) {};
    \node[vertex, label=above:$2$] (2) at (2, 2) {};
    \node[vertex, label=above:$3$] (3) at (4, 2) {};

    \node[vertex, label=below:$a$] (a) at (0, 0) {};
    \node[vertex, label=below:$b$] (b) at (2, 0) {};
    \node[vertex, label=below:$c$] (c) at (4, 0) {};

    \draw[red!30, line width=0.5 mm]    (b) -- (1);
    \draw[blue!30, line width=0.5mm]   (c) -- (1);
    \draw[red!30, line width=0.5mm]    (a) -- (2);
    \draw[green!30, line width=0.5mm]  (c) -- (2);
    \draw[blue!30, line width=0.5mm]   (a) -- (3);
    \draw[green!30, line width=0.5mm]  (b) -- (3);

    \draw[green, line width=1 mm]  (a) -- (1);
    \draw[blue, line width=1 mm]   (b) -- (2);
    \draw[red, line width=1 mm]    (c) -- (3);
\end{tikzpicture}
\end{minipage}
\caption{Correspondence between a Latin square and a properly edge-colored complete bipartite graph, on the left, and between a transversal and rainbow perfect matching, on the right.}
    \label{fig:enter-label}
\end{figure}
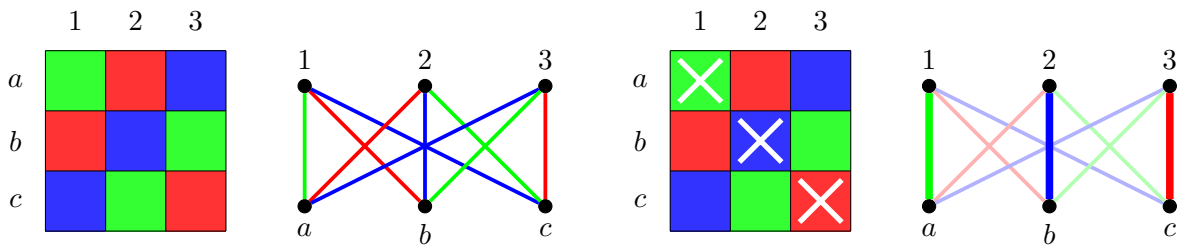

In this short survey we discuss the following general question: Given a particular class of edge-colored graphs, what rainbow or other edge-constrained
subgraphs  do they contain? We present several results in this area and demonstrate their applications to problems
in graph decomposition, additive combinatorics, theoretical computer science, and
coding theory. It is of course impossible to cover everything in such a short article, and therefore the
choice of results we present is inevitably biased. Yet, we hope to describe enough
examples and problems from this fascinating subject to appeal to
researchers not only in discrete mathematics but in other areas as well.

\section{Cycles with edge-coloring constraints}
The study of cycles in graphs has long been one of the central topics in graph theory. A basic result in this area says that a maximal acyclic graph is a tree, i.e., an $n$-vertex graph with $n$ edges contains a cycle. Finding cycles in an appropriately defined auxiliary graph often corresponds to finding a certain dependency between the objects studied. In many situations this is too simplistic since one can not capture all relevant properties of the underlying objects by a simple graph. As we will see, in some cases encoding the relations among studied objects via a properly edge-colored graph gives one significantly more power. In this situation however, one usually needs to find a cycle with some special properties. This leads us to the following very natural question. How many edges in a properly edge-colored graph will guarantee that there exists a cycle satisfying certain edge-coloring constraints?

One of the most basic constraints is to ask that a cycle have a color which appears uniquely, namely a color that is distinct from the colors of all the remaining edges of this cycle. This is clearly a very special case of the rainbow cycle problem in which all edges must have a unique color. This problem was first proposed in 2007 by Keevash, Mubayi, Sudakov and Verstra\"ete \cite{KMSV}, who observed that, somewhat surprisingly, a linear number of edges is no longer enough to guarantee such a cycle. The construction which shows this uses Cayley sum graphs together with their canonical edge coloring.

\begin{definition}
    Given a group $\Gamma$ and a subset $S \subset \Gamma$, let $G$  be the graph whose vertices are the elements of $\Gamma$. Two vertices $x, y$ are adjacent if and only if $x+y=s$ for some $s \in S$. Moreover, $G$ has a canonical proper edge coloring in which we color edge $(x,y)$ by color $s=x+y$. 
\end{definition}

To get an example of a properly edge-colored graph without a cycle with a unique color consider the 
Cayley sum graph $G$ of the Abelian group $\mathbb{F}_2^k$ with respect to the standard basis $S=\{e_1, e_2, \ldots e_k\}$, where $e_i$ is a vector whose $i$-th coordinate is one and the rest are zero. This graph has $n=2^k$ vertices, it is $k$-regular and it is properly edge-colored with the color of the edge $(x,y)$ being $i$ if $x+y=e_i$. Hence, $G$ has $nk/2=\frac{1}{2}n\log_2n$ edges. Moreover, the sum of the basis vectors corresponding to colors of the edges of a cycle in $G$ equals twice the sum of elements corresponding to the vertices of this cycle. In other words, since we are working over $\mathbb{F}_2$, the colors of any cycle in $G$ correspond to the collection of basis vectors whose sum is zero. Since these vectors are linearly independent we have that each basis vector (or each color) should appear in every cycle an even number of times. Thus no cycle has an edge of unique color and we have a construction which shows that more than $\frac{1}{2}n\log_2n$ edges are needed in order to guarantee a cycle with a unique color in a graph.

\begin{figure}[ht]

    \centering
    \begin{minipage}{0.4\textwidth} 
        \centering
\begin{tikzpicture}[scale=0.45,line width=2pt]

\draw[red] (-2.88,-4.12)-- (0.04,-5);
\draw[blue]  (0.04,-5)-- (2.9804315517300366,-4.190887962275627);
\draw[red]  (2.9804315517300366,-4.190887962275627)-- (0.06043155173003667,-3.3108879622756273);
\draw[blue]  (0.06043155173003667,-3.3108879622756273)-- (-2.88,-4.12);
\draw[green]  (-2.88,-4.12)-- (-1.5388377953303893,-1.3810067651113584);
\draw[blue] (1.3811622046696106,-2.2610067651113583)-- (4.321593756399647,-1.4518947273869856);
\draw[red]  (4.321593756399647,-1.4518947273869856)-- (1.4015937563996472,-0.5718947273869857);
\draw[blue]  (1.4015937563996472,-0.5718947273869857)-- (-1.5388377953303893,-1.3810067651113584);
\draw[green] (1.3811622046696106,-2.2610067651113583)-- (0.04,-5);
\draw[green]  (4.321593756399647,-1.4518947273869856)-- (2.9804315517300366,-4.190887962275627);
\draw[green] (1.4015937563996472,-0.5718947273869857)-- (0.06043155173003667,-3.3108879622756273);
\draw[red]  (-1.5388377953303893,-1.3810067651113584)-- (1.3811622046696106,-2.2610067651113583);

\draw [fill=black] (-2.88,-4.12) circle (4pt);
\draw [fill=black]  (0.04,-5) circle (4pt);
\draw [fill=black]  (0.06043155173003667,-3.3108879622756273) circle (4pt);
\draw [fill=black]  (-1.5388377953303893,-1.3810067651113584) circle (4pt);
\draw [fill=black] (1.3811622046696106,-2.2610067651113583) circle (4pt);
\draw [fill=black]  (1.4015937563996472,-0.5718947273869857) circle (4pt);
\draw [fill=black]  (2.9804315517300366,-4.190887962275627) circle (4pt);
\draw [fill=black]  (4.321593756399647,-1.4518947273869856) circle (4pt);
\end{tikzpicture}
\end{minipage}%
    \hfill
    \begin{minipage}{0.6\textwidth} 
        \captionof{figure}{
           Cayley sum graph of the group $\mathbb{F}_2^3$ with respect to the standard basis and its canonical edge coloring.
        }
        \label{fig:cayley-sum-graph}
    \end{minipage}

\end{figure}
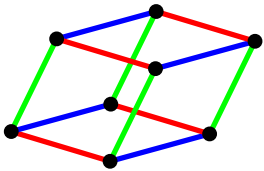

The above bound has the right order of magnitude. This was shown in \cite{KMSV},
using an argument which counts the number of rainbow paths in the edge-colored graph. Moreover the proof produces a cycle with an edge of unique color that 
has at most logarithmic length.

\begin{lemma}
\label{cycle_unique_color}
    Any properly edge-colored $n$-vertex graph with at least $2n \log_2 n$ edges contains a cycle of length at most $2\log_2 n$ with an edge of unique color.
\end{lemma}

\begin{proof}
Let $G$ be a properly edge-colored $n$-vertex graph  with at least $2n \log_2 n$ edges. By deleting vertices of low degree we can assume that the minimum degree of $G$ is at least $r=2\log_2 n$. Let $\ell= \log_2 n$ and for the sake of contradiction, suppose $G$ contains no cycle of length at most $2\ell$ with an edge of unique color.
We will double-count the number of rainbow paths in $G$ of length $\ell$. Note that counting rainbow paths is the same as counting rainbow walks, since if the rainbow walk intersects itself we immediately get a desired rainbow cycle. In order to give a lower bound on the number of rainbow paths we can then take a vertex $v_1 \in V(G)$ and greedily count the number of rainbow paths in $G$ of the form $v_1v_2v_3 \ldots v_{\ell+1}$. By the minimum degree assumption and since $G$ is properly edge-colored, we have at least $r$ options for $v_2$, then at least $r-1$ options for $v_3$, as the second edge of the path should have a different color from the first edge. With the same reasoning we have at least $r-(i-2)$ options for $v_i$, concluding that the number of such rainbow paths is at least $n \prod_{j=0}^{\ell-1} (r-j)\geq n\ell^\ell$. 

On the other hand, since there is no cycle with an edge of unique color, we can upper bound the number of rainbow paths as follows. For each pair of vertices $x,y$ (there are at most $n^2$ of them), observe that if there is a rainbow path of length $\ell$ between them using a set of colors $\cal C$, then any other rainbow path between $x$ and $y$ must use the same set of colors $\cal C$. Therefore, there is at most $\ell!$ rainbow paths of length $\ell$ between any pair of vertices. Since $n^2\ell!<n\ell^\ell$ for $\ell= \log_2 n$ and large enough $n$, this contradiction completes the proof.
\end{proof}

\subsection{Even covers and locally decodable codes}

The above simple lemma has interesting applications to problems in theoretical computer science and coding theory. To describe these applications we need the following definition.
\begin{definition}
A non-empty collection of sets $\cal F$ is called an even cover 
if it covers every vertex an even number of times. Equivalently, $\cal F$ is an even cover if 
$\sum_{E \in {\cal F}} v_E = 0$ over $\mathbb{F}_2$, where $v_E$ denotes the characteristic vector of $E$.
The characteristic vector of a subset $E$ is a $0/1$-vector whose coordinates corresponding to the elements of $E$ are $1$ and the rest of whose coordinates are $0$.
\end{definition}

A hypergraph is just a collection of subsets called edges and a hypergraph is $k$-uniform if every subset has size $k$.
When $\cal H$ is a $2$-uniform hypergraph, i.e.\ a usual graph, an even cover is simply an even subgraph, i.e.\ a subgraph whose vertices all have even degree. Such a subgraph is a union of edge-disjoint cycles in $\cal H$, and so an even cover with the smallest number of edges must in fact be a simple cycle. Hence, the length of the smallest even cover equals to the length of the shortest cycle of $\cal H$, which we call girth. The extremal trade-off between the size of $\cal H$ (i.e., the number of edges in the graph) and its girth was conjectured by Bollob\'as and confirmed by Alon, Hoory and Linial~\cite{AlonHL02} who proved the \emph{irregular Moore bound}, which says that every graph on $n$ vertices with average degree $d$ has girth at most $2 \lceil \log_{d-1}(n) \rceil$. It is an outstanding open problem whether the constant $2$ in this bound can be further improved (for the best known constant, see \cite{LUW}). 

Even covers in $k$-uniform hypergraphs also correspond to linearly dependent subsets of a system of linear equations over $\mathbb{F}_2$ which have exactly $k$ nonzero coefficients (such equations are called $k$-sparse). To see why, let us associate a variable $x_v$ with each vertex $v \in V(\cal H)$ and the $k$-sparse equation $\sum_{v \in E} x_v = b_E,  b_E \in \mathbb{F}_2$ with each edge $E$ in $\cal H$. Then, observe that for any even cover $\cal F$, the left hand sides of the equations corresponding to $E \in {\cal F}$ add up to $0$ and are thus linearly dependent. Thus, size vs girth trade-offs for $k$-uniform hypergraphs correspond to the largest possible size $\ell$ of the minimum linear dependency in a system of $k$-sparse linear equations with $m$ equations in $n$ variables. Such $k$-sparse linear equations arise naturally as parity check equations for \emph{low-density parity check} (LDPC) error correcting codes, and the size vs girth trade-offs for $k$-uniform hypergraphs thus correspond to rate vs distance trade-offs for LDPC codes. 

By the equivalence between even covers and linear dependencies, it is clear that every hypergraph with $m \geq n+1$ hyperedges must have an even cover of size at most $n+1$. This is tight as can be seen, for example, by considering a $1$-uniform hypergraph consisting of $n$ singletons. This suggests the following natural question.

\begin{question}
    How does the trade-off between the number of edges $m$ of the $k$-uniform hypergraph $\cal H$ and the maximum possible girth of $\cal H$ look as $m$ increases beyond $n+1$?
\end{question}

For $k$-uniform hypergraphs with $k>2$, the size vs girth trade-offs were first studied by Naor and Verstra\"ete~\cite{NaorV08} through applications to rate vs distance trade-offs for LDPC codes mentioned above. They showed that every $\cal H$ with $m \geq n^{k/2} \log^{O(1)}(n)$ hyperedges on $n$ vertices must contain an even cover of length $O(\log n)$. The $\log^{O(1)}(n)$ factor was further improved to a $O(\log \log n)$-factor in a subsequent work of Feige~\cite{Feige08}. For $k=2$, this recovers a coarse version of the irregular Moore bound. For $k>2$, however, there is an interesting regime between the two extreme thresholds of $m=n+1$ (with maximum possible girth of $n+1$) and $m \sim n^{k/2} \log^{O(1)}(n)$ (with maximum possible girth of $O(\log n)$). 

In 2008, Feige~\cite{Feige08} formulated a conjecture about this in-between regime that suggests a smooth interpolation between the two extremes noted above. 
\begin{conjecture}
    Fix any $k \geq2$. Then, there exists a sufficiently large $C>0$ such that for sufficiently large $n$ and every $\ell$, every $k$-uniform hypergraph $H$ with $m \geq C n (\frac{n}{\ell})^{k/2-1}$ hyperedges has an even cover of length at most $O(\ell \log_2 n)$.
\end{conjecture}

\noindent Random hypergraphs witness the quantitative behavior in Feige's conjecture, up to a multiplicative factor of $\log(n)$ in $m$. Indeed, Feige's conjecture was based on the hypothesis that random hypergraphs are approximately extremal for the purpose of avoiding short even covers. The motivation for this conjecture comes from the question of showing the existence of (and/or efficiently finding) polynomial size \emph{refutation witnesses}. These are easily verifiable certificates of unsatisfiability of randomly chosen $k$-SAT formulas parameterized by the number of clauses. Feige's conjecture implies that the result of Feige, Kim and Ofek~\cite{FeigeKO06}, that random $3$-SAT formulas with $m \geq \Omega(n^{1.4})$ clauses admit a polynomial size refutation witness with high probability, will extend to the significantly more general setting of \emph{smoothed} $3$-SAT formulas, in addition to simplifying the construction and arguments based on the second moment method in~\cite{FeigeKO06}.

Until recently, not much was known about Feige's conjecture except for the work of Alon and Feige~\cite{AlonF09} that showed a suboptimal version for the case of $k=3$ and that of Feige and Wagner~\cite{feige2016generalized} that built an approach to the problem of even covers by viewing them as an instance of generalized girth problems about hypergraphs. In 2022, Guruswami, Kothari and Manohar~\cite{GuruswamiKM22} proved Feige's conjecture up to an additional loss of a $\log^{2k}(n)$ multiplicative factor in $m$ via a rather involved spectral argument applied to the so-called Kikuchi graph, which we discuss in more details below. Their argument was tightened to reduce the loss down to a $O(\log n)$ multiplicative factor in $m$ by Hsieh, Kothari and Mohanty in~\cite{HsiehKM23}. 
Using Lemma \ref{cycle_unique_color}, together with Hsieh, Kothari, Mohanty and Munh\'a Correia \cite{HKMMS}, we can give a substantially simpler, purely combinatorial, argument that recovers their result and improves the logarithmic factors for hypergraphs of odd uniformity. This proof is particularly short in the case of even uniformity, so we present it here.

\begin{theorem}
\label{even-cover}
    For all even $k$ and sufficiently large constant $C$, every $k$-uniform $n$-vertex hypergraph with at least $C n \left(n/\ell \right)^{k/2-1} \cdot \log_2 n$ hyperedges contains an even cover of size $O(\ell \log_2 n)$.
\end{theorem}

\begin{proof}
Let ${\cal H}$ be a hypergraph satisfying the assertion of the theorem. We define an edge-colored graph $G$, called the Kikuchi graph, as follows.
The vertex set of $G$ consists of all $\ell$-element subsets of the vertex set $V({\cal H})$. For two sets $S,T$ of size $\ell$
define an edge $S \xleftrightarrow{} T$  if there exists an edge $E$ of the hypergraph ${\cal H}$ such that characteristic vectors $v_S, v_T$ satisfy $v_S+v_T = v_E$, i.e., 
$|S \cap E| = |T \cap E| = k/2$ and $|S \cap T|=\ell-k/2$. Note that $G$ is a subgraph of the corresponding Caley sum graph. Color the edge $S \xleftrightarrow{} T$ in $G$ with color $E$. An important observation is that this edge coloring is proper. Indeed, given an $\ell$-set $S$ and a color $E \in E({\cal H})$ there exist at most one $\ell$-set $T$ satisfying $v_S+v_T = v_E$, since this implies $v_T=v_S+v_E$.

\noindent 
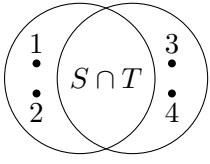
\begin{figure}[ht]
    \centering
    \begin{minipage}{0.3\textwidth} 
        \centering
        \begin{tikzpicture}[scale=0.5]
            \draw (0,0) circle [radius=2cm];

            \draw (1.4,0) circle [radius=2cm];

            \fill (-1.15,0.4) circle [radius=3pt] node[above] {1};
            \fill (-1.15,-0.4) circle [radius=3pt] node[below] {2};

            \fill (2.45,0.4) circle [radius=3pt] node[above] {3};
            \fill (2.45,-0.4) circle [radius=3pt] node[below] {4};

            \node at (0.7,0) {$S \cap T$};
        \end{tikzpicture}
    \end{minipage}%
    \hfill
    \begin{minipage}{0.7\textwidth} 
        \captionof{figure}{
            An illustration of the Kikuchi graph defined above with an edge $E = \{1,2,3,4\}$ of ${\cal H}$ and an edge $S \xleftrightarrow{} T$ of $G$ colored with $E$.
        }
        \label{fig:gadget}
    \end{minipage}
\end{figure}

Next we apply Lemma \ref{cycle_unique_color} to show that $G$ contains a cycle of length $O(\ell \log n)$ with a unique color.
Note that $G$ has $N = {n \choose \ell}$ vertices and each edge $E \in E(\mathcal{H})$ creates at least $ \binom{n-k}{\ell-k/2}$ edges of $G$, since we need to choose the intersection $S \cap T$ which is disjoint from $E$ and combine it with an arbitrary partition of $E$ into two equal parts. Therefore, the number of edges in $G$ is 
$$e(G) \geq e({\cal H}) \binom{n-k}{\ell-k/2}\geq C n \left(\frac{n}{\ell} \right)^{k/2-1} \log n \cdot \binom{n-k}{\ell-k/2}\geq 2N \log_2 N.$$
The last inequality follows from $N = \binom{n}{\ell}$ and $\log_2 N \leq \log_2 n^\ell=\ell \log_2 n$ by a simple but tedious computation. Since $G$ is properly-colored, Lemma \ref{cycle_unique_color} implies that it contains the desired cycle of length at most $r \leq 2\log_2 N \leq 2\ell \log_2 n$.

To finish the proof, take the cycle $S_1,S_2, \ldots, S_r$ in $G$ given above and consider the collection of edges $\{E_i\}$ in ${\cal H}$ which appear as colors in this cycle.
By the definition of $G$, we have $v_{S_i}+v_{S_{i+1}}=v_{E_i}$. Thus $\sum_i v_{E_i}
=\sum_i\big(v_{S_i}+v_{S_{i+1}}\big)=0$. Since some edge $E_i$ appears in this sum uniquely (i.e., does not cancel out) we got a nontrivial even cover of size at most $r = O(\ell \log n)$, as desired.
\end{proof}

The bound in Lemma \ref{cycle_unique_color} on the number of edges which guarantees a cycle with a unique color also has an application to another well known problem related to locally decodable codes.
A binary error correcting code is a map ${\cal C}: \{0,1\}^m \rightarrow \{0,1\}^n$, where we view the input as an $m$-bit ``message'' and the output as an $n$-bit codeword. By a slight abuse of notation, we use ${\cal C}$ to also denote the set of all codewords: $\{y \mid \exists x \in \{0,1\}^m, {\cal C}(x) = y\}$. We say that ${\cal C}$ is linear if, when viewing the input and output as $\mathbb{F}_2^m$ and $\mathbb{F}_2^n$, respectively, ${\cal C}$ is an $\mathbb{F}_2$-linear map. A code ${\cal C}$ is called $(q,\delta)$-locally decodable (LDC) if, in addition, it admits a randomized \emph{local decoding} algorithm. Such a local decoding algorithm takes as input any target message bit $i$ for $1 \leq i \leq m$ and a corrupted codeword $y$ such that $\dist(y,{\cal C}(x)) \leq \delta n$ for some $x \in \{0,1\}^m$, where $\dist$ counts the number of coordinates in which $y$ and $C(x)$ differ. The goal of the algorithm is to access at most $q$ locations in $y$ and output $x_i$ correctly with high probability over the choice of the $q$ locations. In other words, the local decoder can decode any bit of the message by reading at most $q$ locations of the received corrupted codeword. 

Locally decodable codes are intensely investigated in computer science (see the survey~\cite{Yekhanin12} for background and applications) with applications to probabilistically checkable proofs, private information retrieval~\cite{Yekhanin10}, and worst-case to average-case reductions in computational complexity theory. They also have deep connections with additive combinatorics and incidence geometry~\cite{Dvir12}. We are typically concerned with codes that are locally decodable with very few queries, such as $q=2$ or $3$, and the fundamental question is the smallest possible $n = n(m)$ such that there is a $(q,\delta)$-binary LDC ${\cal C}:\{0,1\}^m \rightarrow \{0,1\}^n$. Classical results have essentially completely resolved the case of $q=2$ and we know that codewords must have blocklength of at least $n \geq 2^{\Omega(m)}$~\cite{GoldreichKST06,KerendisW04}, and a matching upper bound is provided by the Hadamard codes (when $\delta$ is constant). The case of $q=3$ already presents wide gaps, where, until recently, the best known lower bound~\cite{GoldreichKST06,KerendisW04} was $n \geq \tilde{O}(m^2)$, while the best known construction~\cite{Yekhanin08,Efremenko09} gives a $3$-query binary linear code with $n \leq \exp (\exp (O(\sqrt{\log m \log \log m})))$. Recently, Alrabiah et.~al.~\cite{AlrabiahGKM23} improved the quadratic bound above to obtain a lower bound of $n \geq m^3/ \mathrm{polylog} m$ for 3-query binary locally decodable codes. As in the case of the results on even covers (mentioned above), their proof involves spectral arguments on signed adjacency matrices of Kikuchi graphs based on matrix concentration inequalities. In \cite{HKMMS}, we use the variant of Lemma \ref{cycle_unique_color} to give a simple and purely combinatorial proof of the following result, which recovers their cubic lower bound and has slightly better logarithmic dependence.

\begin{theorem} \label{thm:LDC-non-even-cover}
Let $C:\mathbb{F}_2^m \rightarrow \mathbb{F}_2^n$ be a linear map that gives a $3$-query locally decodable code with distance $\delta>0$. Then, $m \leq K n^{1/3} \log n$ for $K = 10^7/\delta^2$. 
\end{theorem}

\subsection{Rainbow cycles}
As we already mentioned in the introduction the most natural edge coloring constraint is to ask for a subgraph to be rainbow. In this subsection we discuss a problem of finding rainbow cycles in edge-colored graphs. The study of this question was initiated by Keevash, Mubayi, Sudakov and Verstra\"ete \cite{KMSV}, who asked how many edges one can have in a properly edge-colored $n$-vertex graph without containing a rainbow cycle. The example of the Caley sum graph of $\mathbb{F}_2^k$ with respect to the standard basis, which we described earlier, gives a lower bound for this problem of $\Omega(n\log n)$ and it was conjectured in \cite{KMSV} that this is best possible.

\begin{conjecture}
\label{rainbow}
    There exists a constant $C$ such that every properly colored $n$ vertex graph with at least $Cn \log n$ edges contains a rainbow cycle.
\end{conjecture}

The first nontrivial upper bound for the rainbow cycle problem was obtained by Das, Lee and Sudakov \cite{DLS13}, who showed that, for sufficiently large $n$, having $n^{1+o(1)}$ edges is enough to guarantee such a cycle. This paper also introduced one of the two main approaches to the rainbow cycle problem, in which one passes to an expander subgraph to find in it a rainbow cycle. Expanders are graphs with strong connectivity properties, and they have many applications across mathematics and computer science (see \cite{HLW}). Another important approach, based on the homomorphism counts, was pioneered by Janzer in \cite{Jan20}, who proved the first polylogarithmic bound of $O(n(\log n)^4)$ for this problem.
Keevash, Mubayi, Sudakov and Verstra\"ete \cite{KMSV} also proved that if $G$ is a properly edge-colored $n$-vertex graph with at least $n\log_2(n+3)$ edges, then for some $k$ it contains a cycle of length $k$ which has more than $k/2$ different colors. Because of the Cayley sum construction, described in the beginning of this section, this is tight up to a constant factor. Using a novel variant of the homomorphism count technique Janzer and Sudakov \cite{JS} strengthened this substantially by finding a cycle which is almost rainbow.

\begin{theorem} \label{thm:almost rainbow cycle}
    If $n$ is sufficiently large, $0<\eps<1/2$ and $G$ is a properly edge-colored $n$-vertex graph with at least $\frac{4}{\eps}n\log n$ edges, then for some $k$ it contains a cycle of length $k$ with more than $(1-\eps)k$ different colors.
\end{theorem}

Conjecture \ref{rainbow} is still open. After several further improvements in \cite{Tom22, JS, KLLT}, the current best bound was obtained in an impressive paper of 
Alon, Buci\'c, Sauermann, Zakharov and Zamir \cite{alon2023essentially}, who came very close to proving the conjecture.

\begin{theorem}
\label{best-rainbow}
    There exists a constant $C>0$ such that every properly edge-colored $n$-vertex graph with
average degree at least $C \log n \log \log n$ contains a rainbow cycle.
\end{theorem}
The rainbow cycle problem has interesting applications to questions in other areas.
Next we discuss one such application to a problem in additive combinatorics.
\begin{definition}
Let $(\Gamma,\cdot)$ be a group with identity element $e$. A subset $S \subseteq \Gamma$ is called dissociated if there is no solution to the equation $g_1^{\eps_1}\dotsm g_m^{\eps_m}=e$ with 
$\eps_i=\pm1$, $m\geq 1$ and distinct $g_1,\ldots, g_m \in S$.
\end{definition}
This definition means that there is no nontrivial relation between elements of $S$ with exponents $\pm1$. Therefore,
the maximal dissociated sets play to some extent a similar role in arbitrary groups as linearly independent sets play in vector spaces. 
From this definition, it is easy to see that a maximal dissociated subset of a set $A \subseteq \Gamma$ generates the set $A$, using only exponents $\pm 1$. In addition, one can show \cite{LY} that for abelian groups all maximal dissociated subsets of any set $A \subseteq \Gamma$ have almost the same size.  It is therefore not surprising that controlling the size of a maximal dissociated subset is relevant for various problems in additive number theory, see \cite{BK, Bour, Chang, KS, San1}. Dissociated sets also play an important role in Harmonic analysis, see e.g.\ \cite{Ko,Ru}. 
The above-mentioned properties lead naturally to a notion of dimension of a finite subset $A$ of a  group $\Gamma$ which was defined by Schoen and Skhredov \cite{SS}.
\begin{definition}
    The additive dimension $\dim A$ of subset $A$ of group $\Gamma$ is the size of the largest dissociated subset of $A$.
\end{definition}
A subset $A \subset \Gamma$ has a small doubling if the cardinality of $A \cdot A=\{aa' : a,a' \in A\}$ is of order $O(|A|)$.
The study of the structure of sets with small doubling in the abelian setting traces back to a celebrated result of Freiman \cite{Fr} from 1964 and has been the subject of extensive study ever since, see e.g.\ \cite{GR, San3} and references therein. There has been a lot of work extending results about sets with small doubling to nonabelian groups, with a foundational paper of Tao \cite{Tao} fueling the work in this direction. One can naturally expect that a set with small doubling should be rather structured and could be
``well approximated'' by a ``subspace/subgroup'' generated by its maximal dissociated subset. The following theorem, proved in \cite{San2} (part $1$) and \cite{alon2023essentially} (part $2$),  supports this intuition. 

\begin{theorem}
\label{additive-dimension}
Let $\Gamma$ be a group, and let $A \subseteq \Gamma$ be a finite subset with $|A \cdot A|\le K|A|$ for some positive integer $K$.
\begin{enumerate}
    \item 
    If $\Gamma$ is abelian, then $\dim A \le O_K(\log |A|)$.
    \item 
    For general group $\Gamma$, $\dim A \le O_K(\log^{1+o(1)} |A|)$.
\end{enumerate}
\end{theorem}
The proof of the first statement can be deduced from Lemma \ref{cycle_unique_color} and we omit it here. We illustrate the application of Theorem \ref{best-rainbow} by proving
the second statement.
\begin{proof}
    
Let $A \subseteq \Gamma$ with $|A \cdot A|\le K|A|$ and let $S$ be a dissociated subset of $A$ of maximum size. Then $|S|=\dim A$.
Define an auxiliary bipartite graph $H$ with disjoint parts $A$ and $B=A \cdot S=\{as : a \in A, s \in S\}$
such that $(a,b)$ is an edge of $H$ if $b=as$ for some $s \in S$. Also, color such an edge of $H$ by color $s$.

By definition, this edge coloring is proper and the number of vertices of $H$ satisfies
$|V(H)| \leq|A|+|A \cdot A| \leq O_K(|A|)$. Moreover, the number of edges in $H$ is $|A||S|$. Therefore, if $|S|>C_K \log |A| \log \log |A|$, for a sufficiently large constant $C_K$ depending on $K$, we can use Theorem~\ref{best-rainbow} to find a rainbow cycle in $G$. This cycle has an even length, since $H$ is bipartite.
A rainbow cycle with distinct colors $s_1,s_2,\ldots,s_{2\ell}$ appearing in that order implies $$s_1s_2^{-1}s_3s_4^{-1}\cdots s_{2\ell-1}s^{-1}_{2\ell}=e.$$ 
This contradicts that $S$ is dissociated and shows that $|S| \le C_K \log |A| \log \log |A|= O(\log^{1+o(1)} |A|)$.     
\end{proof}

The rainbow cycle problem has several additional interesting applications. For example, Theorem~\ref{best-rainbow} was recently used by Alrabiah and Guruswami \cite{AG} to obtain nearly tight bounds on the dimension of 3-query locally correctable binary linear codes. These codes are closely related to locally decodable codes, which we already mentioned above. For more definitions and history of the problem, we refer the interested reader to \cite{AG}.

\section{Matchings and transversals in Latin squares}

Recall that a Latin square of order $n$ is an $n\times n$ array filled with $n$ symbols so that every symbol appears only once in each row and in each column. A partial transversal of size $t$ is a collection of $t$ cells of the Latin square which do not share the same row, column or symbol. A transversal (also known as a full transversal) is a partial transversal of order $n$. As we already discussed in the Introduction, the study of transversals in Latin squares goes back to the work of Euler on orthogonal Latin squares.  

It is easy to see that for even $n$ there are many Latin squares without full transversals. For example, consider a Latin square $L=(\ell_{i,j})$ of order $2k$ which is an addition table of the cyclic group $\mathbb{Z}_{2k}$ . Suppose it has a transversal $T=\ell_{1,\sigma(1)}, \ldots, \ell_{2k,\sigma(2k)}$, where $\sigma$ is some permutation of $1, \ldots, 2k$. By definition, $\ell_{i,\sigma(i)}$ is a permutation of the numbers $1, \ldots, 2k$ and also $\ell_{i,\sigma(i)}=i+\sigma(i)\bmod 2k$. This implies that sum of the symbols in the transversal $T$ is
$$\sum_{i=1}^{2k} i=\sum_{i=1}^{2k} \ell_{i,\sigma(i)}=\sum_{i=1}^{2k} (i+\sigma(i))=2  \sum_{i=1}^{2k} i\bmod 2k.$$
As $\sum_{i=1}^{2k} i=k\bmod 2k$, this is a contradiction and shows that the addition table of $\mathbb{Z}_{2k}$ has no transversal.
\begin{figure}[h]
\begin{minipage}{0.39\textwidth}
    \centering
\begin{tikzpicture}[scale=.6]
    \def\size{4}

    \foreach \x in {0, 1, 2, 3} {
        \foreach \y in {0, 1, 2, 3} {
            \draw (\x, -\y) rectangle (\x + 1, -\y - 1);
        }
    }

    \foreach \i in {0, 1, 2, 3} {
        \node at (-0.5, -\i - 0.5) {\(\i\)}; 
        \node at (\i + 0.5, 0.5) {\(\i\)};  
    }

    \foreach \x in {0, 1, 2, 3} {
        \foreach \y in {0, 1, 2, 3} {
            \pgfmathtruncatemacro{\result}{mod(\x + \y, 4)}
            \node at (\x + 0.5, -\y - 0.5) {\(\result\)};
        }
    }

    \draw[thick] (0, 0) -- (\size, 0);
    \draw[thick] (0, 0) -- (0, -\size);
\end{tikzpicture}
\end{minipage}
\begin{minipage}{0.6\textwidth} 

        \captionof{figure}{
           The addition table of the cyclic group $\mathbb{Z}_4$, for which the corresponding Latin square does not contain a full transversal.
        }
        \label{fig:addition table}
    \end{minipage}

\end{figure}

One central problem on Latin squares is to determine which of them have transversals.  This question is very difficult even in the case of multiplication tables of finite groups.
In 1955 Hall and Paige \cite{hall1955complete} conjectured that the multiplication table of a group $G$ has a full transversal exactly if the $2$-Sylow subgroups of $G$ are trivial or non-cyclic. It took 50 years to establish this conjecture. Its proof, which is a combination of works by Wilcox, Evans and Bray, is based on the classification of finite simple groups (see \cite{wilcox2009reduction} and the references therein). Recently two alternative proof of this conjecture, using very different methods,  were found for large groups. 

The first proof was obtained by Eberhard, Manners and Mrazovi\'c \cite{eberhard2020asymptotic} using tools from analytic number theory. 
Remarkably they were also able to obtain a very accurate quantitative bound for the number of transversals in the multiplication tables of groups which satisfy Hall-Paige conjecture. A different proof of this conjecture, this time by combinatorial techniques, was given by M\"uyesser and Pokrovskiy \cite{MP}. They prove a very general theorem about multiplication tables of large groups which gives a solution to several old combinatorial problems in group theory, including the proof of Hall-Page conjecture for large groups.

Another central problem in this area is to determine the size of the largest partial transversal one can find in every Latin square. As we already explained in the introduction, this is equivalent to finding a maximum rainbow matching in a proper edge coloring of the complete bipartite graph $K_{n,n}$ with $n$ colors. The answer to this problem was famously conjectured more than 50 years ago by Ryser, Brualdi and Stein~\cite{ryser1967neuere, stein1975transversals, brualdi1991combinatorial}.
 
\begin{conjecture}
\label{prob-partial-ryser-brualdi-stein}
Every proper edge coloring of the complete graph $K_{n,n}$ with $n$ colors contains a rainbow matching of size $n-1$. Moroever if $n$ is odd it contains a perfect rainbow matching.
\end{conjecture}

Most research towards the Ryser-Brualdi-Stein conjecture has focused on proving better and better lower bounds on the size of a maximum rainbow matching, trying to get as close to $n-1$ as possible. Here Koksma~\cite{koksma1969lower} found rainbow matchings of size $2n/3+O(1)$ and Drake~\cite{drake1977maximal} improved this to $3n/4+O(1)$. The first asymptotic proof of the conjecture was obtained by Brouwer, De Vries, and Wieringa~\cite{brouwer1978lower}, and independently by Woolbright~\cite{woolbright1978n}, who found rainbow matchings of size $n-\sqrt{n}$. This was improved  in 1982 by Shor~\cite{shor1982} to $n-O(\log^2 n)$. His paper had a mistake which was later rectified, using the original approach, by Hatami and Shor \cite{hatami2008lower}. For  nearly forty years, this was the best known bound for the Ryser-Brualdi-Stein conjecture. Recently Keevash, Pokrovskiy, Sudakov and Yepremyan \cite{KPSY} improved this old bound and showed 
that one can find a rainbow matching of size $n-O\big(\frac{\log{n}}{\log{\log{n}}}\big)$ in every proper edge coloring of $K_{n,n}$ which uses $n$ colors. Finally, last year, in a groundbreaking development, Montgomery \cite{M23} proved the Ryser-Brualdi-Stein conjecture for large even $n$.

\begin{theorem}\label{main-ryser} For every sufficiently large $n$, every proper edge coloring of the complete graph $K_{n,n}$ with $n$ colors contains a rainbow matching of size $n-1$.
\end{theorem}
\noindent
While this confirms the conjecture for even $n$, it falls short of the conjecture by an additive constant $1$ when $n$ is odd. This slack is important in many places in the proof and it seems that further ideas will be needed to find a perfect rainbow matching for odd $n$.

The proofs of the recent bounds on the Ryser-Brualdi-Stein conjecture are rather involved so we only briefly discuss some highlights here.
One of the key ideas needed to find transversals of size $n-O(\frac{\log n}{\log\log n})$ in \cite{KPSY} is a claim that a union of a random nearly-perfect matching and an arbitrary regular bipartite graph produces with high probability a good expander, i.e.\ a graph with strong connectivity properties. 

Given a copy of $K_{n,n}$ which is properly colored with $n$ colors, the authors of \cite{KPSY} begin by finding 
a large random nearly-perfect rainbow matching, containing $n-n^{1-\eps}$ edges for some fixed small $\eps>0$. This initial matching is found using the so-called R\"odl nibble \cite{Ro}, which is a celebrated probabilistic techniques with many applications. This matching is then successively modified in several steps, each time getting a rainbow matching that is one edge larger. At each step, given a current rainbow matching $M$, one takes two 
vertices $x,y\in V(K_{n,n})\setminus V(M)$ which are in different parts, and an $x,y$-alternating path $P$, i.e.\ a path such that even edges in this path are in $M$ and odd edges (starting with the edge containing $x$) all have different colors not used on $M$. Then, removing the even edges of $P$ from $M$ and adding the odd edges gives a rainbow matching with one more edge. Let us call $G$ the graph consisting of the edges with the colors not used on $M$ and note that $G$ is regular (since every vertex touches edges of all the colors).
Since $M$ contains an initial random matching, its union with $G$ is an expander. This essentially implies 
 that for every pair of vertices not in the matching and every set of $C\log n/\log\log n$ colors (for some fixed constant $C$) not on $M$ such a path $P$ can be found. Moreover, this property is sufficiently robust that it can be used to make iterative adjustments (with some additional care) until the rainbow matching uses all but $C\log n/\log\log n$ of the original $n$ colors.
 
Note that the error term of order $\Omega(\log n/\log\log n)$ is a natural barrier for this method, since the diameter of expander graph in this proof is at least $\Omega(\log n/\log\log n)$. This means that if $M$ contains more than $n-\Omega(\log n/\log\log n)$, it will not be possible to find a path $P$ described above, since such paths must contain at least $\Omega(\log n/\log\log n)$ colors which do not appear in $M$.

A significant part of the work of Montgomery \cite{M23} is to exploit the structure of the proper edge coloring of $K_{n,n}$ and to find some 
some way to identify approximate algebraic properties of this edge coloring, as well as the introduction of an `addition structure' to use these properties. Recall that if the edge coloring was coming from addition table of a finite abelian group and we had three 
edges $(x,y)$ of color $c_1$ and $(x',y')$ of color $c_2$ and $(x,x')$ of color $c_3$, then $x+y=c_1, x'+y'=c_2$ and $x+x'=c_3$. Therefore the color of the edge
$(y,y')$ is determined and is equal to $c_1+c_2-c_3$. Having many such three edge paths with colors $c_1,c_3, c_2$ for which the color of the fourth edge is the same 
can be used to find an absorbing structure which, by switching the edges in the rainbow matching, can free pairs of vertices (that are no longer covered by rainbow matching).
Such an absorbing structure can be found and set aside, before a large rainbow matching is found disjointly from it, 
using the method of distributive absorption. 

Absorption is a versatile technique for extending approximate results into exact ones, which was first systematically used 
by R\"odl, Ruci\'nski and Szemer\'edi \cite{RRS}.  Such ideas also appeared in earlier works of Erd\H{o}s-Gy\'arf\'as-Pyber \cite{EGP} and Krivelevich \cite{K}. It has since been implemented in many different settings using a variety of creative constructions. Distributive absorption is an efficient and powerful variant of this method, introduced by Montgomery \cite{Mdist}, which can be used to build a global absorbing structure from small local absorbers.  For more details about this very impressive proof we refer the interested reader to the excellent survey by Montgomery \cite{Mbritish}.

A Latin array is  an $n\times n$ square array filled with an arbitrary number of symbols such that no symbol appears twice in the same row or column. Latin arrays are natural extensions of Latin squares, and also have been extensively studied. A familiar example of such an array is a multiplication table for the elements of two subsets of equal size in some group. 
Having a Latin array is equivalent to having a proper edge coloring of the complete bipartite graph $K_{n,n}$ without any restriction on the number of colors.
It is generally believed that extra colors in a proper edge coloring should help to find a large rainbow matchings.
Although there are proper edge colorings of $K_{n,n}$ with strictly more than $n$ colors and not even a single rainbow perfect matching, there are several results which confirm this intuition. 
For example, Montgomery, Pokrovskiy and Sudakov \cite{MPS3} showed that with some very
weak additional constraints one can find many disjoint rainbow perfect matchings. 
\begin{theorem}\label{manysymbols}  
For every $\eps>0$ and sufficiently large $n$ the following holds. Every proper edge coloring of $K_{n,n}$ with at least $\eps n$ colors appearing at most 
$(1-\eps)n$ times contains at least $(1-\eps) n$ edge disjoint perfect rainbow matchings.
\end{theorem}
\noindent
Note that in a proper edge coloring each color can appear on at most $n$ edges. Hence, 
this result suggests that the only Latin arrays without transversals are small perturbations of Latin squares.

The methods used to study the Ryser-Brualdi-Stein conjecture usually also apply to give results about general proper edge colorings of $K_{n,n}$.
For example the authors of \cite{KPSY} showed that any such coloring contains a rainbow matching of size 
$n-O\big(\frac{\log{n}}{\log{\log{n}}}\big)$. As a consequence of Theorem \ref{manysymbols}, we only need to consider colorings which are close to the colorings that use exactly $n$ 
colors. Similarly, the techniques developed by Montgomery \cite{M23} apply to show that every proper edge coloring of the complete graph $K_{n,n}$ contains a rainbow matching of size $n-1$.

We will finish this section by discussing two extensions of the problem of finding large rainbow matchings in properly edge-colored bipartite graphs $K_{n,n}$.

The first extension considers host graphs that are not complete bipartite graphs. This problem was proposed by Aharoni and Berger~\cite{aharoni2009rainbow}, who conjectured that if $G$ is a bipartite multigraph (i.e., more than 1 edge is allowed between each pair of vertices) which is properly colored such that at least $n-1$ colors have at least $n$ edges, then it has a rainbow matching of size $n-1$. In fact, this is a generalisation of Conjecture~\ref{prob-partial-ryser-brualdi-stein}, since a proper $n$ edge coloring of $K_{n,n}$ clearly contains $n-1$ colors of size $n$. 
For the problem of Aharoni and Berger even finding
an approximate solution appeared to be very difficult. It took a sequence of improvements (see the recent survey by Pokrovskiy~\cite{Alexeysurvey} for this history, and further related problems), before Pokrovskiy~\cite{pokrovskiy2018approximate} gave such a solution. He showed that, for each $\eps>0$ and sufficiently large $n$, having $n-1$ colors with at least $(1+\eps)n$ edges is sufficient for a rainbow matching with $n-1$ edges. His original proof was rather involved and a much simpler proof, using a probabilistic sampling trick, was found in \cite{correia2023short}.

For our second extension, we consider what happens if we no longer require that the coloring of $K_{n,n}$ to be proper, asking only that each of $n$ colors appears $n$ times (this is equivalent to what is known as an equi-$n$-square). In 1975, Stein~\cite{stein1975transversals} conjectured that every such coloring has a rainbow matching with $n-1$ edges. This ambitious conjecture was open for 40 years 
but was eventually disproved by Pokrovskiy and Sudakov~\cite{pokrovskiy2019counterexample}, showing that one can not hope for rainbow matching of size larger than
$n-c \log n$. However, there was a belief that every such coloring at least contains a rainbow matching of asymptotically correct size. After a sequence of improving bounds~\cite{stein1975transversals,aharoni2017conjecture,anastos2024notefindinglargetransversals}, Chakraborti, Christoph, Hunter, Montgomery, and Petrov~\cite{chakraborti2024almostfulltransversalsequinsquares} confirmed this by showing that every coloring of $K_{n,n}$ in which each of $n$ colors appears $n$ times contains a rainbow matching with $(1-o(1))n$ edges (and moreover can even be almost decomposed into such matchings).

The result in \cite{chakraborti2024almostfulltransversalsequinsquares} is very recent, and shows that the study of rainbow subgraphs remains an active and exciting area. One very interesting open question is an outstanding conjecture of Stein~\cite{stein1975transversals} from 1975, giving a midway point between his later disproved conjecture and the Ryser-Brualdi-Stein conjecture. It says that, given a coloring of $K_{n,n}$ in which each of the $n$ colors appears $n$ times such that in one of the vertex classes no two edges of the same color meet (i.e., the coloring is proper on one side), is there always a rainbow matching with $n-1$ edges?

\section{Rainbow trees and decomposition problems}
One of the most natural problems in the study of rainbow structures is to determine which graphs appear as rainbow subgraphs in every properly edge-colored complete graph $K_n$. Since for even $n$ one can properly color $K_n$ using only $n-1$ colors, we may ask in general only for rainbow subgraphs with at most $n-1$ edges. This leads to a natural question: which $n$-vertex trees have a rainbow copy in any proper coloring of $K_n$? 

The special case of this question was considered in 1980 by Hahn~\cite{hahn1980jeu} who conjectured that there is always a rainbow $n$-vertex path. 
This was disproved in \cite{maamoun1984problem} by considering a complete graph on the vertex set $\mathbb{F}_2^k$, where the edge $(x,y)$ is colored using the color $x+y\in \mathbb{F}_2^k\backslash \{0\}$. Recall that this graph is precisely the Cayley sum graph of $\mathbb{F}_2^k$ with respect to all non-zero elements of this group, along with its canonical edge coloring. For any path $P=x_1x_2\ldots x_n$ in this coloring it is easy to see that the sum of the colors of the edges of this path is $x_1+x_n\not =0$. But the sum of all the colors (i.e., all non-zero elements of the group) is zero, showing that there is no rainbow path covering all the vertices. Nevertheless, it was widely believed that any properly colored $K_n$ contains a rainbow path covering all but exceptionally few vertices. In particular, Andersen~\cite{andersen1989hamilton} in 1989 conjectured that one can always find a rainbow path covering $n-1$ vertices. The progress on this conjecture was initially slow, despite the efforts of various researchers, see, e.g., \cite{akbari2007rainbow, Gyarf1, Gyarf2, gebauer, chen}.
 Several years ago, the problem was resolved asymptotically by Alon, Pokrovskiy and Sudakov~\cite{alon2016random}, who showed that any properly colored $K_n$ contains both a rainbow path and a rainbow cycle with $n-O(n^{3/4})$ vertices (this error term was slightly improved in \cite{BM}).

For a general tree, until recently, there were no results that showed the existence of rainbow copies of large trees in properly edge-colored complete graphs. The above example for the $n$-vertex path can be extended (see \cite{BenzingTrees}) to show that there are proper edge colorings of $K_n$ which do not contain rainbow copies of certain spanning trees. However, it is still possible that nearly-spanning trees exist in all proper colorings of complete graphs. This follows from the following general result, proved by Montgomery, Pokrovskiy and Sudakov \cite{MPS1}. 
Recall that an edge-coloring of $K_n$ is
locally $k$-bounded if each vertex is in at most $k$ edges of any one color. Therefore, a proper coloring is exactly a locally $1$-bounded coloring. 

\begin{theorem}\label{almostspan}
For $\eps>0,k\in \mathbb{N}$ and sufficiently large $n$, any locally $k$-bounded edge-coloring of $K_n$ contains a rainbow copy of every tree with at most $(1-\eps)n/k$ vertices.
\end{theorem}

Since a locally $k$-bounded edge-coloring of $K_n$ may have only $(n-1)/k$ distinct colors, Theorem~\ref{almostspan} is tight up to the constant $\eps$ for each $k$. One of the distinguishing features of this result is that it places no conditions on the trees other than the number of vertices they may have. In comparison, all of the previous graph packing and labeling results required a maximum degree assumption, which was conjectured to be unnecessary. 
An important tool in the proof of this theorem is to demonstrate that, in properly colored complete graphs, a large rainbow matching can typically be found in a random vertex set using a random set of colors chosen with the same density. In fact, such a matching can cover almost all the vertices in the random vertex set. This allows most of a large tree to be embedded if it can be decomposed into certain large matchings. Contrastingly, one needs deterministic methods to embed vertices in the tree with high degree. These techniques for embedding trees with high degree vertices may offer new approaches to related problems.

\subsection{Graph decompositions, labelings and orthogonal double-covers.}
Theorem~\ref{almostspan} has applications in several different areas of graph theory. We now discuss three such applications to graph decompositions, graph labelings, and orthogonal double-covers. With each application this theorem proves an asymptotic form of a well-known conjecture. In each case, we consider some special edge coloring of $K_n$, coming from a graph-theoretic problem, to which we apply Theorem~\ref{almostspan}.

\medskip
\noindent
\textbf{Graph decompositions.}\,
We say that a graph $G$ has a decomposition into copies of a graph $H$ if the edges of $G$ can be partitioned into edge-disjoint subgraphs isomorphic to $H$.
In 1847, Kirkman studied decompositions of complete graphs $K_n$ and showed that they can be decomposed into copies of a triangle if, and only if, $n\equiv 1 \text{ or } 3 \pmod 6$. Wilson~\cite{wilson} generalized this result by proving necessary and sufficient conditions for a complete graph $K_n$ to be decomposed into copies of \emph{any} graph, for large $n$.
A very old problem in this area, posed in 1853 by Steiner, says that, for every $k$, modulo an obvious divisibility condition every sufficiently large complete $r$-uniform hypergraph can be decomposed into edge-disjoint copies of a complete $r$-uniform hypergraph on $k$ vertices. This problem was the so-called ``existence of designs'' question and has practical relevance to experimental designs. It was resolved only very recently in spectacular
work by Keevash \cite{keevash} (see the subsequent work of \cite{glock-all} for an alternative proof of this result). Over the years graph and hypergraph decomposition problems have been extensively studied and
by now this has become a vast topic with many exciting results and conjectures (see, for example, \cite{gallian2009dynamic,wozniak2004packing,yap1988packing}).

Decompositions of complete graphs into large trees, whose size is comparable with the size of the complete graph (in contrast with the existence of designs mentioned above), also has a long history. The earliest such result  was obtained more than a century ago by Walecki. In 1882, he proved that a complete graph $K_n$ on an even number of vertices can be partitioned into edge-disjoint Hamilton paths (this is a path which visits every vertex of the graph exactly once). Since paths are a very special kind of tree it is natural to ask which other large trees can be used to decompose a complete graph. 
One of the oldest and best known conjectures in this area, posed by Ringel in 1963 \cite{ringel1963theory},
deals with the decomposition of complete graphs into edge-disjoint copies of a tree.

\begin{conjecture} 
\label{ringelconj}
Any tree with $n+1$ vertices packs $2n+1$ times into the complete graph $K_{2n+1}$.
\end{conjecture}
\noindent
Note that a complete graph $K_{2n+1}$ has precisely $(2n+1) \cdot n$ edges. Therefore it is conceivable to partition it into $2n+1$ copies of the same tree which has $n$ edges.

This conjecture is known for many very special classes of trees such as caterpillars, trees with $\leq 4$ leaves, diameter $\leq 5$ trees,   trees with $\leq 35$ vertices, and several other classes of trees (see Chapter 2 of \cite{gallian2009dynamic} and the references therein).
There are also some partial general results in the direction of Conjecture~\ref{ringelconj}. Typically, for these results, an extensive technical method is developed which is capable of almost-packing any appropriately-sized collection of certain sparse graphs, see,
e.g., \cite{bottcher2016approximate, messuti2016packing, ferber2017packing, kim2016blow}.   In particular, Joos, Kim, K{\"u}hn and Osthus~\cite{JKKO} proved the above conjecture for very large bounded-degree trees. Ferber and Samotij~\cite{ferber2016packing} obtained an almost-perfect packing of almost-spanning trees with maximum degree $O(n/\log n)$ (giving what is known as an approximate version of the conjecture).
A different proof of the approximate version of Ringel's conjecture for trees with maximum degree $O(n/\log n)$ was obtained by
Adamaszek, Allen, Grosu, and Hladk{\'y}~\cite{adamaszek2016almost}, using graph labelings. 
Theorem \ref{approxringelkotzig} can be used to give the first asymptotic solution for Ringel's conjecture applicable with no degree restriction.

\begin{theorem}\label{approxringelkotzig}
For $\eps>0$ and sufficiently large $n$, any $(n+1)$-vertex tree packs at least $2n+1$ times into the complete graph $K_{(2+\eps)n}$.
\end{theorem}

\begin{figure}[hbt!]
\vspace{-0.4cm}

  \centering
\begin{tikzpicture}

\def\vxrad{0.1cm};
\def\circrad{1.5};
\def\shift{-6};
\def\medline{0.04cm};
\def\tline{0.04cm};
\def\tvxrad{0.075cm};
\def\incircrad{1};

\begin{scope}[shift={(\shift,0)}]

\foreach \x in {0,...,8}
{
\draw coordinate (B\x) at  (40*\x:\incircrad);
}

\draw ($0.9*(B4)+0.9*(B5)$) node {\large $T:$};

\draw [line width=\tline] (B0) -- (B5);
\draw [line width=\tline] (B4) -- (B5);
\draw [line width=\tline] (B3) -- (B5);
\draw [line width=\tline] (B0) -- (B6);

\foreach \x in {0,3,4,5,6}
{
\draw [fill=black] (B\x) circle [radius=\tvxrad];
}

\end{scope}

\foreach \x in {0,...,8}
{
\draw coordinate (A\x) at  (40*\x:\circrad);
}

\draw ($0.8*(A4)+0.8*(A5)$) node {\large $K_9:$};

\foreach \x in {0,...,7}
{
\pgfmathtruncatemacro\y{\x+1};
\draw [line width=\medline,darkgreen!50] (A\x) -- (A\y);
}
\foreach \x in {8}
{
\pgfmathtruncatemacro\y{\x-9+1};
\draw [line width=\medline,darkgreen!50] (A\x) -- (A\y);
}

\foreach \x in {0,...,6}
{
\pgfmathtruncatemacro\y{\x+2};
\draw [line width=\medline,blue!50] (A\x) -- (A\y);
}
\foreach \x in {0,...,5}
{
\pgfmathtruncatemacro\y{\x+3};
\draw [line width=\medline,black!50] (A\x) -- (A\y);
}
\foreach \x in {0,...,4}
{
\pgfmathtruncatemacro\y{\x+4};
\draw [line width=\medline,red!50] (A\x) -- (A\y);
}

\foreach \x in {7,...,8}
{
\pgfmathtruncatemacro\y{\x-9+2};
\draw [line width=\medline,blue!50] (A\x) -- (A\y);
}
\foreach \x in {6,...,8}
{
\pgfmathtruncatemacro\y{\x-9+3};
\draw [line width=\medline,black!50] (A\x) -- (A\y);
}
\foreach \x in {5,...,8}
{
\pgfmathtruncatemacro\y{\x-9+4};
\draw [line width=\medline,red!50] (A\x) -- (A\y);
}

\draw [line width=0.1cm,red] (A0) -- (A5);
\draw [line width=0.1cm,darkgreen] (A4) -- (A5);
\draw [line width=0.1cm,blue] (A3) -- (A5);
\draw [line width=0.1cm,black] (A0) -- (A6);

\foreach \x in {0,...,8}
{
\draw [fill=black] (A\x) circle [radius=\vxrad];
\draw ($1.2*(A\x)$) node {\x};
}

\end{tikzpicture}

\vspace{-0.3cm}
  \caption{The ND-coloring of $K_9$ and a rainbow copy of a tree $T$ with four edges. The color of each edge corresponds to its Euclidean length.  By taking cyclic shifts of this tree around the centre of the picture we obtain $9$ disjoint copies of the tree decomposing $K_9$ (and thus a proof of Ringel's Conjecture for this particular tree). To see that this gives 9 disjoint trees, notice that each edge must be shifted to the other edges of the same color (since shifts are isometries).}\label{FigureIntro}
  
  \vspace{-0.1cm}
  
\end{figure}
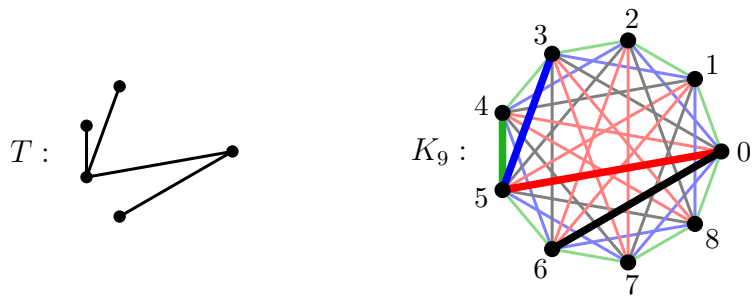

To see the connection between Theorem~\ref{almostspan} and Conjecture~\ref{ringelconj} consider the following edge-coloring of the complete graph with vertex set $\{0,1,\dots,2n\}$.
This coloring, which we call the near distance (ND-) coloring,
was introduced by Kotzig~\cite{rosa1966certain} to study Ringel's conjecture. Color the edge $ij$ by color $k$, where $k\in [n]$, if either $i=j+k$ or $j=i+k$ with addition modulo $2n+1$. As we explain below, it is easy to show (see also Figure \ref{FigureIntro}) that if an $(n+1)$-vertex tree has a rainbow embedding into the ND-coloring of $K_{2n+1}$ then Conjecture~\ref{ringelconj} holds for that tree.

\begin{proof}[Proof of Theorem~\ref{approxringelkotzig}] Let $\ell=(1+\eps/3)n$.
	Consider the ND-coloring of the complete graph $K_{2\ell+1}$, defined above. This is a locally 2-bounded coloring of $K_{2\ell+1}$, which thus, by Theorem~\ref{almostspan}, contains a rainbow copy of $T$, $S_0$ say. Now, for each $i\in [\ell]$, let $S_i$ be the tree with vertex set $\{v+i:v\in V(S_0)\}$ and edge set
	$\{\{v+i,w+i\}:vw\in E(S_0)\}$.

	Note that the color of $vw\in E(K_{2\ell+1})$ is the same as the color of the edge $\{v+i,w+i\}$ for each $i\in [2\ell]$, and under the translation $x\mapsto x+1$ the edge $vw$ moves around all $2\ell+1$ edges with the same color. Thus, each tree $S_i$ is rainbow, and all the trees $S_i$ are disjoint.
\end{proof}

\noindent
\textbf{Graph labeling.}\,
Graph labeling originated in methods introduced by Rosa~\cite{rosa1966certain} in 1967 as a potential path towards proving Ringel's conjecture. In the intervening decades, a large body of work has steadily developed concerning different labelings.  Labeled graphs serve as useful models for a broad range of applications in coding theory,
circuit design, communication networks, database management, secret sharing schemes, cryptology, and models for constraint programming over finite domains (see the extensive survey \cite{gallian2009dynamic}). One old, well-known, conjecture in this area concerns harmonious labelings. This type of labeling was introduced by Graham and Sloane~\cite{GS80} and arose naturally out of the study of additive bases. Given an Abelian group $\Gamma$ and a graph $G$, we say that a labeling $f:V(G)\to \Gamma$ is $\Gamma$-harmonious if the map $f':E(G)\to \Gamma$ defined by $f'(xy)=f(x)+f(y)$ is injective.
In the case when $\Gamma$ is a group of integers modulo $n$ we omit it from our notation and simply call such a labeling harmonious.
In the particular case of an $n$-vertex tree, Graham  and Sloane asked for a harmonious labeling using $\mathbb{Z}_{n-1}$ in which each label is used on some vertex, so that a single label is used on two vertices -- where this exists we call the tree harmonious. More generally, we also call a graph with $m$ edges and at most $m$ vertices harmonious if it has an injective harmonious labeling with $\mathbb{Z}_m$.
According to an unpublished result of Erd\H{o}s~\cite{GS80}, almost all graphs are not harmonious. On the other hand, Graham and Sloane~\cite{GS80} in 1980 made the following conjecture for trees.

\begin{conjecture}
All trees are harmonious.
\end{conjecture}
This conjecture is known for many special classes of trees such as caterpillars, trees with at most $31$ vertices, and some other classes of trees (see Chapter 2 of \cite{gallian2009dynamic} and the references therein). Moreover, 
{\.Z}ak  \cite{zak2009harmonious} proposed an asymptotic weakening of this conjecture. He asked to prove that every tree has an injective $\mathbb{Z}_{n+o(n)}$-harmonious labeling.

Note that, for any injective labeling of the vertices of the complete graph by elements of an Abelian group, the edge-coloring which is obtained by taking sums of labels of vertices is proper. Therefore we can use  Theorem~\ref{almostspan} to study such colorings. In particular, we can obtain the following general result which shows that every tree is almost harmonious. If $\Gamma$ is taken to be the cyclic group, this theorem proves {\.Z}ak's conjecture from~\cite{zak2009harmonious}.

\begin{theorem}\label{thmharlabel}
Every $n$-vertex tree $T$ has an injective $\Gamma$-harmonious labeling for any Abelian group $\Gamma$ of order $n+o(n)$.
\end{theorem}

\begin{proof}[Proof of Theorem~\ref{thmharlabel}] Suppose $|\Gamma|=\ell= (1+\eps)n$ for some $\eps>0$ and let $T$ be a tree on $n$ vertices. Identify the vertices of $K_\ell$ with the elements of $\Gamma$ and consider an edge-coloring that colors the edge $ij$ by $i+j$. This is a proper coloring, so, when $n$ is large, by Theorem~\ref{almostspan} it contains a rainbow copy of $T$, which corresponds to a harmonious labeling. By taking $\Gamma=\mathbb{Z}_{\ell}$, we deduce that for any $\eps>0$ there exists $n_0$ such that any tree with $n\geq n_0$ vertices has an injective harmonious labeling with at most $(1+\eps)n$ labels.
\end{proof}

\noindent
\textbf{Orthogonal double covers.}\,
Theorem~\ref{almostspan} can also be used to obtain an asymptotic solution for another old graph decomposition problem. An  orthogonal double cover of the complete graph $K_n$ by a graph $G$ is a collection $G_1,\ldots,G_n$ of subgraphs of $K_n$ such that each $G_i$ is a copy of $G$, every edge of $K_n$ belongs to exacly two of the copies and any two copies have exactly one edge in common. The study of orthogonal double covers was originally motivated by problems in statistical design theory (see \cite{dinitz1992contemporary}, Chapter 2).
Since the number of copies in the double cover is $n$, it follows that $G$ must have $n-1$ edges.
The central problem here is to determine for which graphs $G$ there is a an orthogonal double cover of $K_n$ by $G$. In full generality, this extends several well-known problems, such as the existence question for biplanes or symmetric $2$-designs (see, e.g., \cite{hughes1978biplanes}). About 20 years ago, Gronau, Mullin, and Rosa \cite{gronau1997orthogonal} made the following conjecture about trees.
\begin{conjecture} \label{Conjecture_Orthogonal}
Let $T$ be an $n$-vertex tree which is not a path on $3$ edges. Then, $K_n$ has an orthogonal double cover by copies of $T$.
\end{conjecture}
This conjecture is known to hold for certain classes of trees including trees with diameter $\leq 3$ and trees with $\leq 13$ vertices (see \cite{gronau1997orthogonal, leck1997orthogonal} for a full list).
To see the connection between Conjecture~\ref{Conjecture_Orthogonal} and Theorem~\ref{almostspan}, we will consider a Caley sum coloring of a complete graph on $2^k$ vertices, where edges are colored by the sum of their endpoints in the abelian group $\mathbb{Z}_2^k$. By considering such a coloring we can show that Conjecture~\ref{Conjecture_Orthogonal} is asymptotically true whenever $n$ is a power of $2$.
\begin{theorem}\label{orthogonal}
Let $n=2^k$ and let $T$ be a tree on $n-o(n)$ vertices. Then $K_n$ contains $n$ copies of $T$ such that
every edge of $K_n$ belongs to at most two copies and any two copies have at most one edge in common.
\end{theorem}
\begin{proof}
Identify $V(K_n)$ with the group $\mathbb{Z}_2^k$. Color each edge $ij$ using the color $i+j\in \mathbb{Z}_2^k$. By Theorem~\ref{almostspan}, $K_n$ has a rainbow copy $S$ of $T$.  For all $x\in \mathbb{Z}_2^k$, define a permutation $\phi_x:V(K_n)\to V(K_n)$ by $\phi_x(v)=x+v$ (with addition in $\mathbb{Z}_2^k$).
Use $\phi_x(S)$ to denote the subgraph of $K_n$ with edges $\{\phi_x(a)\phi_x(b): ab\in E(S)\}$. Notice that, since $\phi_x$ is a permutation of $V(K_n)$, $\phi_x(S)$ is a tree isomorphic to $T$. We claim that the family of $n$ trees $\{\phi_x(S): x\in \mathbb{Z}_2^k\}$ satisfies the theorem.

Notice that since $\phi_x(a)+\phi_x(b)=a+x+b+x=a+b$, the permutations $\phi_x$ preserve the colors of edges.
This implies that the trees $\phi_x(S)$ are all rainbow.
Notice that the only edges fixed by the permutations $\phi_x$ are those colored by $x$.
Finally, notice that $\phi_{x-y}\circ \phi_y=\phi_{x}$. Combining these, we have that if a color $c$ edge is in two trees $\phi_x(S)$ and $\phi_y(S)$, then $x-y=c$ in $\mathbb{Z}_2^k$. This implies that the trees $\phi_x(S)$ cover any edge at most twice, and any pair of them have at most one edge in common.
\end{proof}

\noindent
\textbf{Ringel's conjecture revisited.}\,
Recall the near distance (ND-)coloring, which was used to prove the asymptotic version of Ringel's conjecture.
It colors the edge $ij$ of the complete graph $K_{2n+1}$ by color $k$, where $k\in [n]$, if either $i=j+k$ or $j=i+k$ with addition modulo $2n+1$.
In the application of Theorem \ref{almostspan} we only used  the fact that this coloring is $2$-bounded to find a rainbow copy of an arbitrary tree of size $(1-\eps)n$. But the (ND-)coloring has a lot of structural properties and is very symmetric. So it is natural to ask whether one can exploit the structure of 
the (ND-)coloring to embed in it rainbow trees with exactly $n$ edges. Using several critical new ideas combined with the approach from Theorem~\ref{almostspan}, this was 
done by Montgomery, Pokrovskiy and Sudakov~\cite{MPS2}. 

\begin{theorem}
\label{ringel}
For every sufficiently large $n$, the complete graph $K_{2n+1}$ can be decomposed into copies of any tree with $n$ edges.
\end{theorem}
\noindent
This theorem constructs a decomposition of the complete graph by considering cyclic shifts of one copy of a given tree (see Figure~\ref{FigureIntro}). The existence of such a cyclic decomposition was separately conjectured by Kotzig~\cite{rosa1966certain}. Therefore, this also gives a proof of the conjecture by Kotzig for large trees. A different proof of Ringel's conjecture for large $n$ was obtained in \cite{KeevashStaden}.

The proof of Theorem \ref{ringel} in~\cite{MPS2} is rather involved for discussing in detail. It introduces three key new methods. Firstly, when dealing with very high degree vertices, it uses a completely deterministic approach for finding a rainbow copy of the tree. This approach heavily relies on the features of the ND-colouring. 
Secondly, it uses techniques to embed 99\% of any tree $T$ randomly into the ND-colouring in a rainbow fashion. The techniques in~\cite{MPS2} are based on the probabilistic methods in~\cite{MPS1}, except we need the random embedding to maintain some degree of independence between the vertices and the colors used for this embedding. High degree vertices in $T$ force dependencies between the vertices and the colors used on any randomized embedding of large subgraphs of $T$. These dependencies were a major barrier to generalising previous random approaches to Ringel's Conjecture to trees with high degree vertices. Maintaining some independence between vertices and colors here is a subtle task and a critical part of the proof.

Finally, having embedded 99\% of the tree $T$, the embedding in~\cite{MPS2} is completed using absorption. This is a particularly challenging task, and so we prepare for it during our initial embedding by carefully constructing an absorbing structure which assists with the precise task of completing the embedding while using each unused color exactly once. The novelty in~\cite{MPS2}, compared to the classical applications of the absorption method, is in the interplay of our constructions with the randomized parts of our embedding, the constructions themselves, and the use of absorption to embed graphs with high degree vertices. 

\bigskip
\noindent 
{\bf Acknowledgments.} The author would like to express his gratitude to Aleksa Milojevi\'c for many thoughtful discussions, 
carefully reading the manuscript, and drawing several illustrations. We also wish to thank Richard Montgomery and Matija Buci\'c for 
their helpful remarks and suggestions.

\end{document}